\documentclass[11pt,reqno,a4paper]{amsart}

\usepackage{mumedial}

\author{Quentin M\'erigot}
\email{Quentin.Merigot@sophia.inria.fr}
\address{INRIA Sophia-Antipolis\\
2004 route des Lucioles\\
06560 Sophia-Antipolis \\
France
}

\title[Size of the medial axis and stability of curvature measures]
{Size of the medial axis and stability of Federer's curvature measures}

\usepackage[a4paper]{geometry}

\begin{document}

\begin{abstract}
In this article, we study the $(d-1)$-volume and the covering numbers
of the medial axis of a compact subset of $\RR^d$. In general, this
volume is infinite; however, the $(d-1)$-volume and covering numbers
of a filtered medial axis (the $\mu$-medial axis) that is at distance
greater than $\eps$ from the compact set can be explicitely
bounded. The behaviour of the bound we obtain with respect to $\mu$,
$\eps$ and the covering numbers of $K$ is optimal.

From this result we deduce that the projection function on a compact
subset $K$ of $\RR^d$ depends continuously on the compact set $K$, in
the $\LL^1$ sense. This implies in particular that Federer's curvature
measure of a compact subset of $\RR^d$ with positive reach can be
reliably estimated from a Hausdorff approximation of this subset,
regardless of any regularity assumption on the approximating subset.
\end{abstract}

\maketitle

\section{Introduction}

We are interested in the following question: given a compact set $K$
with positive reach, and a discrete approximation, is it possible to
approximate Federer's curvature measures of (see \cite{federer1959cm}
or \S \ref{sec:federer} for a definition) knowing the discrete
approximation only ? A positive answer to this question has been given
in \cite{ccsm2009boundary} using convex analysis. In this article, we
show that such a result can also be deduced from a careful study of
the ``size'' --- that is the covering numbers --- of the medial axis.

The notion of medial axis\footnote{also known as ambiguous locus in
  Riemannian geometry} has many applications in computer science. In
image analysis and shape recognition, the skeleton of a shape is often
used as an idealized version of the shape \cite{sonka1999image}, that
is known to have the same homotopy type as the original shape
\cite{lieutier2004aob}. In the reconstruction of curves and surfaces
from point cloud approximations, the distance to the medial axis
provides a estimation of the size of the local features that can be
used to give sampling conditions for provably correct reconstruction
\cite{amenta1999srv}. The flow associated with the distance function
$\d_K$ to a compact set $K$, that flows away from $K$ toward local
maxima of $\d_K$ (that lie in the medial axis of $K$) can be used for
shape segmentation \cite{chazal2009analysis}. The reader that is
interested by the computation and stability of the medial axis with
some of these applications in mind can refer to the survey
\cite{attali2007stability}.

The main technical ingredient needed for bounding the covering numbers
of the subsets of the medial axis that we consider is a Lipschitz
regularity result for the so-called normal distance to the medial
axis.  When $K$ is a compact submanifold of class $\Class^{2,1}$, this
function is globally Lipschitz on any $r$-level set of the distance
function to $K$, when the radius $r$ is small enough
\cite{itoh2001lipschitz, li2005distance, castelpietra2009regularity}.
When $K$ is the analytic boundary of a bounded domain $\Omega$ of
$\RR^2$, the normal distance to the medial axis of $\partial \Omega$
is $2/3$-H\"older on $\Omega$ \cite{cannarsa2007holder}.

However, without strong regularity assumption on the compact set $K$,
it is hopeless to obtain a global Lipschitz regularity result for
$\tau_K$ on a parallel set of $K$. Indeed, such a result would imply
the finiteness of $(d-1)$-Hausdorff measure of the medial axis, which
is known to be false --- for instance, the medial axis of a generic
compact set is dense.

We show however, that the normal distance to the medial axis is
Lipschitz on a suitable subset of a parallel set. This enables us to
prove the following theorem on the covering numbers of the
$\mu$-medial axis (see \S \ref{sec:mumedial} for a definition):

\begin{theoremA}
For any compact set $K\subseteq \RR^d$, a parameter $\eps$ smaller
than the diameter of $K$, and $\eta$ small enough,
\begin{equation*}
\LebNum\left(\Medial_\mu(K) \cap
(\RR^d\setminus K^{\eps}), \eta\right) \leq
 \LebNum(\partial K, \eps/2)
\BigO\left(\left[\frac{\diam(K)}{\eta\sqrt{1-\mu}}\right]^{d-1}\right) 
\end{equation*}
\end{theoremA}

From this theorem, we deduce a quantitative Hausdorff-stability
results for projection function, which is the key to the stability of
Federer's curvature measure (see Proposition \ref{prop:wass}):

\begin{theoremB}
Let $E$ be a bounded open set of $\RR^d$. The application that
maps a compact subset of $\RR^d$ to the projection function $\p_K\in
\LL^1(E)$ is locally $h$-H\"older, for any exponent $h$ smaller than
$1/(4d-2)$.
\end{theoremB}

Note that a similar result with a slightly better H\"older exponent
has been obtained in \cite{ccsm2009boundary}. However, the proofs in
this article give are very different and give a more geometric insight
on the Hausdorff-stability of projection functions. Nonetheless, the
main contribution of this article lies in
Theorem~\ref{th:covering:mumedial}.

\section{Boundary measures and medial axes}

\subsection{Distance, projection, boundary measures}
Throughout this article, $K$ will denote a compact set in the
Euclidean $d$-space $\RR^d$, with no additional regularity assumption
unless specified otherwise. The \emph{distance function} to $K$,
denoted by $\d_K: X \to \RR^+$, is defined by $d_K(x) = \min_{p\in K}
\nr{p-x}$. A point $p$ of $K$ that realizes the minimum in the
definition of $\d_K(x)$ is called an \emph{orthogonal projection} of
$x$ on $K$. The set of orthogonal projections of $x$ on $K$ is denoted
by $\proj_K(x)$.

The locus of the points $x \in \RR^d$ which have more than one
projection on $K$ is called the \emph{medial axis} of $K$.  Denote
this set by $\Medial(K)$.  For every point $x$ of $\RR^d$ not lying in
the medial axis of $K$, we let $\p_K(x)$ be the unique orthogonal
projection of $x$ on $K$. This defines a map $\p_K: \RR^d
\setminus{\Medial(K)} \to K$, which we will refer to as the
\emph{projection function} on the compact set $K$.

\begin{figure}
\centering
\includegraphics[width=\textwidth]{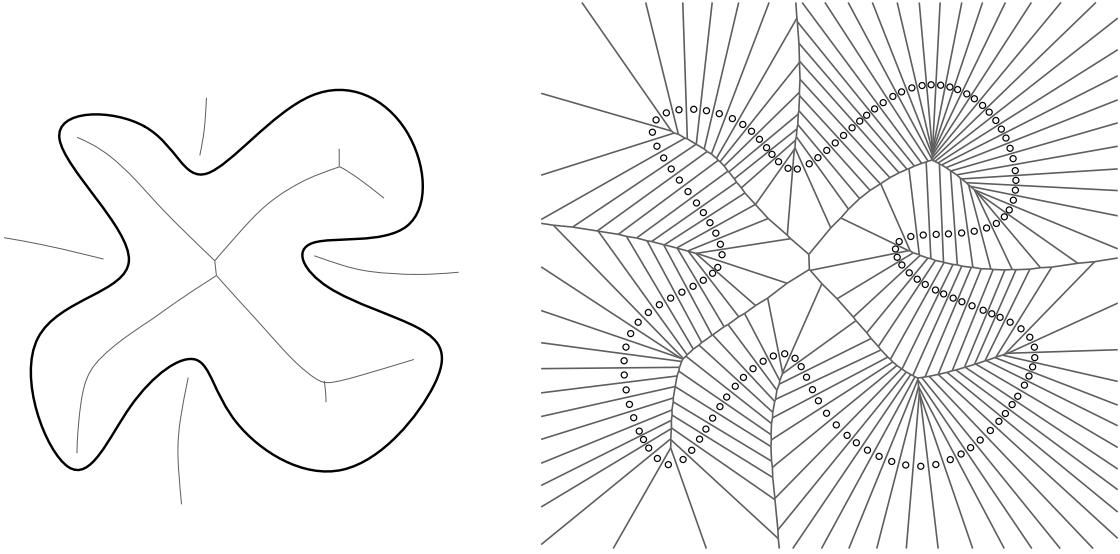}
\caption{Medial axis of a curve $C$ in the plane, and Voronoi diagram
  of a point cloud $P$ sampled on the curve.
}
\end{figure}

\begin{definition}
Let $K$ be a compact subset and $E$ be a measurable subset of
$\RR^d$. We will call \emph{boundary measure} of $K$ with respect to
$E$ the pushforward of the restriction of the Lebesgue measure to $E$
on $K$ by the projection function $\p_K$, or more concisely $\mu_{K,E}
= {\p_K}_\# \restr{\Haus^d}{E}$.
\end{definition}

We will be especially interested in the case where $E$ is of the form
$K^r$, where $K^r$ denotes the $r$-tubular neighborhood of $K$,
i.e. $K^r = \d_K^{-1}([0,r])$.

\begin{example}[Steiner-Minkowski]
If $P$ is a convex solid polyhedron of $\RR^3$, $F$ its set of faces,
$E$ its set of edges and $V$ its set of vertices, then the following
formula holds:
$$
\mu_{P,P^r} = \restr{\Haus^3}{P}  
+ r \sum_{f \in F}
\restr{\Haus^2}{f} \\
+ r^2 \sum_{e \in E} K(e) \restr{\Haus^1}{e} + 
r^3 \sum_{v\in V} K(v) \delta_v
$$ where $K(e)$ is the angle between the normals of the faces adjacent
to the edge $e$, and $K(v)$ the solid angle formed by the normals of
the faces adjacent to the vertex $v$.

For a general convex polyhedra the measure $\mu_{K,K^r}$ can similarly
be written as a sum of weighted Hausdorff measures supported on the
$i$-skeleton of $K$, whose local density is the local external
dihedral angle.
\end{example}

\begin{example}[Weyl]
Let $M$ be a compact smooth hypersurface of $\RR^d$, and denote by
$\sigma_i(p)$ is the $i$th elementary symmetric polynomial of the
$(d-1)$ principal curvatures of $M$ at a point $p$ in $M$. Then, for
any Borel subset $B$ of $\RR^d$, and $r$ small enough, the
$\mu_{K,K^r}$-measure of $B$ can be written as
$$ \mu_{K,K^r}(B) = \sum_{i=0}^{d-1} \const(i,d) \int_{B \cap M}
\sigma_i(p) \d M(p).$$ This formula can be generalized to submanifolds
of any codimension \cite{weyl1939vt}.
\end{example}

\subsection{Federer curvature measures and reach}
\label{sec:federer}
Following Federer \cite{federer1959cm}, we will call \emph{reach} of a
compact subset $K$ of $\RR^d$ the smallest distance between $K$ and
its medial axis, i.e. $\reach(K) = \min_{x \in \Medial(K)}
\d_K(x)$.

Generalizing Steiner-Minkowski and Weyl tubes formula, Federer proved
that as long as $r$ is smaller than the \emph{reach} of $K$, the
dependence in $r$ of the boundary measure $\mu_{K,K^r}$ is a
polynomial in $r$, of degree bounded by the ambient dimension $d$:

\begin{theorem}[Federer]
For any compact set $K \subseteq \RR^d$ with reach greater than $R$,
there exists $(d+1)$ uniquely defined (signed) measures $\Phi_{K}^0$,
\dots,$\Phi_{K}^d$ supported on $K$ such that for any $r\leq R$,
$$\mu_{K,K^r} = \sum_{i=0}^d \omega_{d-i} \Phi_{K,i} r^i$$ where
$\omega_k$ is the volume of the $k$-dimensional unit sphere.
\end{theorem}

These measures are uniquely defined and Federer calls them
\emph{curvature measures} of the compact set $K$.

\subsection{Stability of boundary and curvature measures}
The question of the stability of boundary measures is a particular
case of the more general question of geometric inference. Given a
(discrete) approximation of a compact subset $K$ of $\RR^d$, what
amount of geometry and topology of $K$ is it possible to recover ?
In our case, the question is to bound the Wasserstein distance between
the boundary measures of two compact subsets as a function of their
Hausdorff distance.

Recall that the \emph{Hausdorff distance} between two compact subsets
$K$ and $K'$ is defined by $\dH(K,K') = \nr{\d_K - \d_{K'}}_\infty$.
%
The \emph{Wasserstein distance} (with exponent one) between two
measures $\mu$ and $\nu$ with finite first moment on $\RR^d$ is
defined by
$ \Wass_1(\mu, \nu) = \min_{X,Y} \Expect[\nr{X-Y}] $ where the
minimum is taken over all the couples of random variables $X,Y$ whose
law are $\mu$ and $\nu$ respectively.

\begin{proposition}
Let $E$ be an open subset and $K, K'$ be two compact subsets
of $\RR^d$. Then,
$$\Wass_1\left(\frac{\mu_{K,E}}{\Haus^d(E)},
\frac{\mu_{K',E}}{\Haus^d(E)}\right)
\leq \frac{1}{\Haus^d(E)} \nr{\p_K - \p_{K'}}_{\LL^1(E)}$$
\label{prop:wass}
\end{proposition}

\begin{proof}
See \cite[Proposition 3.1]{ccsm2009boundary}.
\end{proof}

Hence, in order to obtain a Hausdorff stability result for boundary
measures, one only needs to obtain a bound of the type $ \nr{\p_K -
  \p_{K'}}_{\LL^1(E)} = o(\dH(K,K'))$. The possibility to estimate
Federer's curvature measures from a discrete approximation can also be
deduced from a $\LL^1$ stability result for projection functions (see
\cite[\S4]{ccsm2009boundary}).

\section{A first non-quantitative stability result}

Intuitively, one expects that the projections $p_K(x)$ and $p_{K'}(x)$
of a point $x$ on two Hausdorff-close compact subsets can differ
dramatically only if $x$ lie close to the medial axis of one of the
compact sets. This makes it reasonable to expect a $\LL^1$ convergence
property of the projections.  However, since the medial axis of a
compact subset of $\RR^d$ is generically dense (see
\cite{zamfirescu2004cut} or \cite[Proposition~I.2]{these}),
translating the above intuition into a proof isn't completely
straightforward.

\subsection{Semi-concavity of $\d_K$ and $\mu$-medial axis}
\label{sec:mumedial}
The semi-concavity of the distance function to a compact set has been
remarked and used in different contexts \cite{fu1985tne,
  petrunin,cannarsa2004semiconcave,lieutier2004aob}.  More precisely,
we will use the fact that for any compact subset $K\subseteq \RR^d$,
the squared distance function to $K$ is $1$-concave. This is
equivalent to the function $v_K: \RR^d \to \RR, x \mapsto \nr{x}^2 -
\d^2_K(x)$ being convex. Thanks to its semiconcavity one is able to
define a notion of generalized gradient for the distance function
$\d_K$, that is defined even at points where $\d_K$ isn't
differentiable.

Given a compact set $K \subseteq \RR^d$, the subdifferential of the
distance function to $K$ at a point $x \in \RR^d$ is by definition the
set of vectors $v \in \RR^d$ such that $\d_K^2(x+h) \leq \d_K^2(x) +
\sca{h}{v} - \lambda \nr{h}^2$ for all $h \in \RR^d$. The
subdifferential of $\d_K$ at a point $x$ is denoted by $\partial_x
\d_K$, it is the convex hull of the set $\left\{  (p - x)/\nr{p - x} \tq
p \in \proj_K(x) \right\}$.

The gradient $\nabla_x \d_K$ of the distance function $\d_K$ at a
point $x \in \RR^d$ is defined as the vector of $\partial_x \d_K$
whose Euclidean norm is the smallest, or equivalently as the
projection of the origin on $\partial_x \d_K$ (see \cite{petrunin} or
\cite{lieutier2004aob}).  Given a point $x \in \RR^d$, denote by
$\gamma_K(x)$ the center and $\MR{r}_K(x)$ the radius of the smallest
ball enclosing the set of orthogonal projections of $m$ on $K$. Then,
\begin{equation}
\begin{split}
\nabla_x \d_K &= \frac{x - \gamma_K(x)}{\d_K(x)}\\
\nr{\nabla_x \d_K} &= \left(1 -
\frac{\MR{r}^2_K(x)}{\d^2_K(x)}\right)^{1/2} = \cos(\theta)
\end{split}
\label{eq:mudef}
\end{equation}
where $\theta$ is the (half) angle of the cone joining $m$ to
$\B(\gamma_K(m), r_K(m))$

\subsection{$\mu$-Medial axis of a compact set}
The notion of $\mu$-medial axes and $\mu$-critical point of the
distance function to a compact subset $K$ of $\RR^d$ were introduced
by Chazal, Cohen-Steiner and Lieutier in \cite{chazal2006stc}. We
recall the definitions and properties we will need later.

A point $x$ of $\RR^d$ will be called a $\mu$-critical point for the
distance function to $K$(with $\mu \geq 0$), or simply a
$\mu$-critical point of $K$ if for every $h \in \RR^d$,
$$\d_K^2(x+h) \leq \d_K^2(x) + \mu\nr{h} \d_K(x) + \nr{h}^2.$$ The
point $x$ is $\mu$-critical iff the norm of the gradient
$\nr{\nabla_x \d_K}$ is at most $K$.
The \emph{$\mu$-medial axis} $\Medial_\mu(K)$ of a compact set $K
\subseteq \RR^d$ is the set of $\mu$-critical points of the distance
function. Is is easily seen that the medial axis is the union of all
$\mu$-medial axes, with $0\leq\mu<1$: $$\Medial(K) = \bigcup_{0\leq\mu
  < 1} \Medial_\mu(K).$$ Moreover, from the lower semicontinuity of
the map $x \mapsto \nr{\nabla_x \d_K}$, one obtains that for every
$\mu<1$, the $\mu$-medial axis $\Medial_\mu(K)$ of $K$ is a compact
subset of $\RR^d$. The main result of \cite{chazal2006stc} that we
will use is the following quantitative critical point stability
theorem.

\begin{theorem}[Critical point stability theorem]
\label{th:crit-stab}
Let $K,K'$ be two compact sets with $\dH(K,K')\leq \eps$. For any point
$x$ in the $\mu$-medial axis of $K$, there exists a point $y$ in the
$\mu'$-medial axis of $K'$ with 
$\mu' = \mu+2\sqrt{\eps/\d_K(x)}$ and $\nr{x - y} \leq 2 \sqrt{\eps\d_K(x)}$.
\end{theorem}

\subsection{A first non-quantitative stability result}
The goal of this paragraph is to prove the following non-quantitative
$\LL^1$ convergence result for projections:

\begin{proposition}
\label{prop:proj-stab-nonquant}
If $(K_n)$ Hausdorff converges to a compact $K \subseteq \RR^d$, then
for any bounded open set $E$, $\lim_{n\rightarrow +\infty}
\nr{\p_{K_n} - \p_{K}}_{\LL^1(E)} = 0$.
\end{proposition}

In order to do so, for any $L > 0$, and two compact sets $K$ and $K'$,
we will denote $\Delta_L(K,K')$ the set of points $x$ of $\RR^d
\setminus (K \cup K')$ whose projections on $K$ and $K'$ are at least
at distance $L$, i.e. $\nr{\p_K(x) - \p_{K'}(x)} \geq L$. For
technical reasons, we remove all points of the medial axes of $K$ and
$K'$ from $\Delta_L(K,K')$. Since the Lebesgue measure both medial
axes vanishes, this does not affect the measure of $\Delta_L(K,K')$ .

 A consequence of the critical point
stability theorem is that $\Delta_L(K,K')$ lie close to the
$\mu$-medial axis of $K$ for a certain value of $\mu$ (this Lemma is
similar to \cite[Theorem 3.1]{chazal2008nca}):

\begin{lemma}
\label{lem:delta-included}
Let $L > 0$ and $K,K'$ be two compact sets and $\delta \leq L/2$
denote their Hausdorff distance. Then for any positive radius $R$, one
has
$$ \Delta_L(K,K') \cap K^R \subseteq \Medial_\mu(K)^{2\sqrt{R
    \delta}} $$
with
$$\mu = \left(1 + \left[\frac{L-\delta}{4R}\right]^2\right)^{-1/2} + 4
\sqrt{\frac{\delta}{L}}$$
\end{lemma}

\begin{proof}
Let $x$ be a point in $\Delta_L(K,K')$ with $\d_K(x) \leq R$, and
denote by $p$ and $p'$ its projections on $K$ and $K'$
respectively. By assumption, $\nr{p - p'}$ is at least $L$. We let $q$
be the projection of $p'$ on the sphere $\Sph(x, \d_K(x))$, and let
$K_0$ be the union of $K$ and $q$. By hypothesis on the Hausdorff
distance between $K$ and $K'$, the distance between $p'$ and $q$ is at
most $\delta$. Hence, $\dH(K, K_0)$ is at most $2\delta$.

By construction, the point $x$ has two projections on $K_0$, and must
belong to the $\mu_0$-medial axis of $K_0$ for some value of
$\mu$. Letting $m$ be the midpoint of the segment $[p,q]$, we are able
to upper bound the value of $\mu_0$:
$$\mu^2_0 \leq \nr{\nabla_x \d_{K_0}}^2 \leq \cos\left(\frac{1}{2}
\angle(p - x, q - x)\right)^2 = \nr{x - m}^2/\nr{x - p}^2$$
Since $p,
q$ belong to the sphere $\B(x, \d_K(x))$, one has $(p-q) \perp (m -
x)$ and $\nr{x - p}^2 = \nr{x - m}^2 + \frac{1}{4} \nr{p - q}^2$. This gives
$$ \mu_0 \leq \left(1 + \frac{1}{4}\frac{\nr{p - q}^2}{\nr{x -
    m}^2}\right)^{-1/2} \leq \left[1+ \left(\frac{L - \delta}{2
    R}\right)^2\right]^{-1/2}$$
To get the second inequality we used $\nr{x-m} \leq R$ and $\nr{p-q}
\geq L - \delta$.

In order to conclude, one only need to apply the critical point
stability theorem (Theorem \ref{th:crit-stab}) to the compact sets $K$
and $K_0$ with $\d_H(K,K_0) \leq 2\delta$. Since $x$ is in the
$\mu_0$-medial axis of $K_0$, there should exist a point $y$ in
$\Medial_\mu(K)$ with $\nr{x-y}\leq2 \sqrt{R \delta}$ and $\mu = \mu_0
+ 4\sqrt{\delta/L}$.
\end{proof}

\begin{proof}[Proof of Proposition \ref{prop:proj-stab-nonquant}]
Fix $L > 0$, and suppose $K$ and $K'$ are given. One can
decompose the set $E$ between the set of points where the projections
differ by at least $L$ (\ie $\Delta_{L}(K,K') \cap E$) and the
remaining points. This gives the bound:
$$ \nr{\p_{K'} - \p_{K}}_{\LL^1(E)} \leq L \Haus^d(E) +
\Haus^d(\Delta_{L}(K,K') \cap E) \diam(K \cup K') $$

Now, take $R = \sup_E \nr{\d_K}$, so that $E$ is contained in the
tubular neighborhood $K^R$, and fix $L = \eps/\Haus^d(E)$. Then, for
$\delta = \dH(K,K')$ small enough (e.g. less than some $\delta_0$),
the value of $\mu$ given in Lemma \ref{lem:delta-included} is smaller
than one. Denote by $\mu_0$ the value given by the lemma for
$\delta_0$. Then
\begin{equation}
\nr{\p_{K'} - \p_{K}}_{\LL^1(E)} \leq \eps +
\Haus^d(\Medial_{\mu_0}(K)^{2 \sqrt{R\delta}}) \diam(K \cup K')
\label{eq:maj-proj}
\end{equation}
Being compact, $\Medial_{\mu_0}(K)$ is the intersection of its tubular
neighborhoods. Combining this with the outer-regularity of the
Lebesgue measure gives: $$\lim_{\delta \rightarrow 0}
\Haus^d(\Medial_{\mu_0}(K)^{2 \sqrt{R\delta}}) =
\Haus^d(\Medial_{\mu_0}(K)) = 0.$$
Putting this limit in equation
\eqref{eq:maj-proj} concludes the proof.
\end{proof}

\section{Size and volume of the $\mu$-medial axis}

From the proof of Proposition \ref{prop:proj-stab-nonquant}, one can
see that a way to get a quantitative stability of the projection
functions is to control the volume of tubular neighborhoods of some
part of the $\mu$-medial axis. Recall that the \emph{$\eps$-covering
  number} of a subset $X \subseteq \RR^d$ is the minimum number $N$ of
points $x_1, \hdots, x_N$ such that $X$ is contained in the union of
balls $\cup_{i=1}^N \bar{\B}(x_i, \eps)$. The following inequality is then 
straightforward:
\begin{equation}
\Haus^d(X^\eps) \leq \Haus(\B(0,\eps)) \LebNum(X,\eps)
\end{equation}
Our goal in this section is to obtain a bound on the covering numbers
of the considered part of the $\mu$-medial axis (see Theorem
\ref{th:covering:mumedial}) that will allow to control the growth of
the volume of its tubular neighborhoods.


Because of its compactness, one could expect that the $\mu$-medial
axis of a well-behaved compact set will have finite
$\Haus^{d-1}$-measure. This is not the case in general: if one
considers a ``comb'', \ie an infinite union of parallel segments of
fixed length in $\RR^2$, such as $\MC{C} = \cup_{i\in \NN^*}
[0,1]\times\{2^{-i}\} \subseteq \RR^2$ (see Figure \ref{fig:comb}),
the set of critical points of the distance fonction to $\MC{C}$
contains an imbricate comb. Hence $\Haus^{d-1}(\Medial_\mu(\MC{C}))$
is infinite for any $\mu > 0$.

However, for any positive $\eps$, the set of points of the
$\mu$-medial axis of $\MC{C}$ that are $\eps$-away from $\MC{C}$ (that
is $\Medial_\mu(\MC{C}) \cap \RR^d\setminus \MC{C}^\eps$) only contains a
finite union of segments, and has finite $\Haus^{d-1}$-measure.  The
goal of this section is to prove (quantitatively) that this remains
true for any compact set. Precisely, we have:

\begin{figure}
\centering\includegraphics[height=3cm]{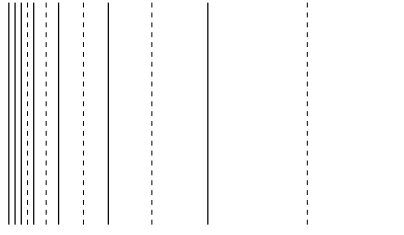}\\
\caption{The ``comb'' and a part of its medial axis (dotted)}
\label{fig:comb}
\end{figure}

\begin{theorem}
\label{th:covering:mumedial}
For any compact set $K\subseteq \RR^d$, $\eps \leq \diam(K)$, and
$\eta$ small enough,
\begin{equation*}
\LebNum\left(\Medial_\mu(K) \cap
(\RR^d\setminus K^{\eps}), \eta\right) \leq
 \LebNum(\partial K, \eps/2) \BigO\left(\left[\frac{\diam(K)}{\eta\sqrt{1-\mu}}\right]^{d-1}\right) 
\end{equation*}
In particular, one can bound the $(d-1)$-volume of the $\mu$-medial axis
\begin{equation*}
\Haus^{d-1}\left(\Medial_\mu(K) \cap (\RR^d\setminus K^{ \eps})\right)
\leq \LebNum(\partial K, \eps/2)
\BigO\left(\left[\frac{\diam(K)}{\sqrt{1-\mu}}\right]^{d-1}\right)
\end{equation*}
\end{theorem}

\begin{remarkst}[Sharpness of the bound] Let $x, y$ be two points at distance
  $D$ in $\RR^d$ and $K = \{x, y\}$. Then, $\Medial(K)$ is simply the
  medial hyperplane between $x$ and $y$. A point $m$ in $\Medial(K)$
  belongs to $\Medial_\mu(K)$ iff the cosine of the angle $\theta =
  \frac{1}{2}\angle(x-m,y-m)$ is at most $\mu$.
\begin{equation*}
 \cos^2(\theta) = 1 - \frac{\nr{x - y}^2}{\d_K^2(m)} = 1 -
\frac{\diam(K)^2}{4 \d_K^2(m)}
\end{equation*}
 Hence, $\cos(\theta) \geq \mu$ iff
$\d_K(m) \leq \frac{1}{2} \diam(K)/\sqrt{1 - \mu^2}$. Let $z$ denote the
midpoint between $x$ and $y$; then $\d_K(m)^2 = \nr{z - m}^2 +
\diam(K)^2/4$. Then, $\Medial_\mu(K)$ is simply the intersection of
the ball centered at $z$ and of radius $\diam(K) \sqrt{\mu^2/(1-\mu^2)}$
with the medial hyperplane. Hence,
\begin{equation*}
\Haus^{d-1}(\Medial_\mu(K)) =
\Omega\left(\left[\frac{\diam(K)\mu^2}{\sqrt{1-\mu^2}}\right]^{d-1}\right)
\end{equation*}
This shows that the behaviour in $\diam(K)$ and $\mu$ of the theorem
is sharp as $\mu$ converges to one.
\end{remarkst}

\subsection{Outline of the proof}
In order to obtain the bound on the covering numbers of the the
$2\eps$-away $\mu$-medial axis $\Medial_\mu(K) \cap (\RR^d\setminus
K^{2\eps})$ given in Theorem \ref{th:covering:mumedial}, we prove that
this set can be written as the image of a part of the level set
$\partial K^\eps$ under the so-called \emph{normal projection on the
  medial axis} $\ell: \RR^d\setminus K \to \MedialClos(K)$.

The main difficulty is to obtain a Lipschitz regularity statement for
the restriction of the map $\ell$ to a suitable subset of $\partial
K^\eps$. There is no such statement for the whole surface $\partial
K^\eps$ in general. However, we are able to introduce a subset
$S_\mu^\eps \subseteq \partial K^\eps$ whose image under $\ell$ cover
the $\eps$-away $\mu$-medial axis, and such that the restriction of
$\ell$ to $S_\mu^\eps$ is Lipschitz. This is enough to conclude.

\subsection{Covering numbers of the  $\mu$-medial axis}
\label{ssec:volumemumedial}

We now proceed to the proof of Theorem \ref{th:covering:mumedial}. 

\begin{definition}
\label{def:tau}
For any point $x \in \RR^d$, we define the \emph{normal distance of
  $x$ to the medial axis} as $\tau_K(x) := \inf \{ t \geq 0\tq x + t
\nabla_x \d_K \in \Medial(K) \}$. We will set $\tau_K(x)$ to zero at
any point in $K$ or in the medial axis $\Medial(K)$.

For any time $t$ smaller than $\tau(x)$, we denote by $\Psi_K^t(x)$
the point $\Psi_K^t(x) = x + t \nabla_x \d_K$.  Finally, for any $x
\not\in K$, we let $\ell_K(x)$ be the first intersection of the
half-ray starting at $x$ with direction $\nabla_x \d_K$ with the
medial axis. More precisely, we define $\ell_K(x) =
\Psi_K^{\tau(x)}(x) \in \MedialClos(K)$.
\end{definition}

\begin{lemma}
Let $m$ be a point of the medial axis $\Medial(K)$ with $\d(x,K) >
\eps$, and $x$ be a projection of $m$ on $\partial K^\eps$. Then
$\ell(x) = m$. 
\label{lemma:psi}
\end{lemma} 

\begin{proof}
By definition of $K^\eps$, $d(m,K) = d(m,K^\eps) + \eps$, so that the
projection $p$ of $x$ on $K$ must also be a projection of $m$ on
$K$. Hence, $m, x$ and $p$ must be aligned. Since the open ball  $B(m,
d(m,p))$ does not intersect  $K$, for any point $y \in  ]p,m[$ the
ball $B(y, d(y,p))$ intersects $K$ only at $p$. In particular, by
definition of the gradient, $\nabla_x \d_K$ must be the unit vector
directing $]p,m[$, \ie $\nabla_x \d_K = (m-x)/d(m,x)$. Moreover, since
$[x,p[$ is contained in the complement of the medial axis,
$\tau(x)$ must be equal to $\d(x,m)$. Finally one gets $\Psi^{\tau(x)}(x)
= x + \d(x,m) \nabla_x \d_K = m$. 
\end{proof}

This statement means in particular that $2\eps$-away medial axis, that
is $\Medial(K) \cap (\RR^d\setminus K^\eps)$, is contained in the
image of the piece of hypersurface $\{x \in \partial K^\eps\tq
\tau_K(x) \geq \eps\}$ by the map $\ell$.

Recall that the radius of a set $K \subseteq \RR^d$ is the radius of
the smallest ball enclosing $K$, while the diameter of $K$ is the
maximum distance between two points in $K$. The following inequality
between the radius and the diameter is known as Jung's theorem
\cite{jung}:
$\radius(K) \sqrt{2 (1 +
  1/d)}  \leq \diam(K)$.

\begin{lemma}
\label{lemma:opposite-point}
For any point $m$ in the $\mu$-medial axis $\Medial_\mu(K)$, there
exists two projections $x,y \in \proj_K(m)$ of $m$ on $K$ such that
the cosine of the angle $\frac{1}{2}\angle(x-m,y-m)$ is smaller than
$\left(\frac{1+\mu^2}{2}\right)^{1/2}$.
\end{lemma}

\begin{proof}
We use the characterization of the gradient of the distance function
given in equation \eqref{eq:mudef}. If $\B(\gamma_K(m), \MR{r}_K(m))$
denotes the smallest ball enclosing $\proj_K(m)$, then $\mu^2 \leq 1 -
\MR{r}^2_K(m)/\d^2_K(m)$.  Using Jung's theorem and the definition of
the diameter, there must exists two points $x,y$ in $\proj_K(m)$ whose
distance $r'$ is larger than $\sqrt{2} \MR{r}_K(m)$. The following
bound on the cosine of the angle $\theta=\frac{1}{2}\angle(x-m,y-m)$
concludes the proof:
\begin{equation}
\begin{split}
\cos^2(\theta) = 1 - \frac{(r'/2)^2}{d^2_K(m)}
\leq  1 - \frac{1}{2}\frac{\MR{r}^2_K(m)}{d^2_K(m)} \leq (1 +
\mu^2)/2
\label{eq:opposite-point}
\end{split}
\end{equation}
\end{proof}

\begin{lemma}
\label{lemma:dist-K-mumedial}
The maximum distance from a point in $\Medial_\mu(K)$ to
$K$ is bounded by $\frac{1}{\sqrt{2}}\diam(K)/\left(1-\mu^2\right)^{1/2}$
\end{lemma}

\begin{proof}
Let $x,y$ be two orthogonal projections of $m \in \Medial_\mu(K)$ on
$K$ as given by the previous lemma. Then, using equation
\eqref{eq:opposite-point}, one obtains
$$1 - \frac{\nr{x - y}^2/4}{\d_K^2(m)} \leq
(1+\mu^2)/2.$$ Hence, $\d_K^2(m) \leq \frac{1}{2}(1-\mu^2)^{-1}
\nr{x - y}^2$, which proves the result.
\end{proof}

Let us denote by $S_\mu^\eps$ the set of points $x$ of the
hypersurface $\partial K^\eps$ that satisfies the three conditions
below:
\begin{itemize}
\item[(i)] the normal distance to the medial axis is bounded below:
  $\tau(x) \geq \eps$ ;
\item[(ii)] the image of $x$ by $\ell$ is in the $\mu$-medial axis of $K$:
$\ell(x) \in \Medial_\mu(K)$;
\item[(iii)] there exists another projection $y$ of $m = \ell(x)$ on
  $\partial K^\eps$ with
  $$\cos\left(\frac{1}{2}\angle(p-m,q-m)\right) \leq
  \sqrt{\frac{1+\mu^2}{2}}$$
\end{itemize}

A reformulation of Lemmas \ref{lemma:opposite-point} and 
\ref{lemma:psi} is the following  corollary:
\begin{corollary}
The image of $S_\mu^\eps$ by the map $\ell$ covers the whole
$2\eps$-away $\mu$-medial axis: $\ell(S_\mu^\eps) = \Medial_\mu(K)
\cap (\RR^d \setminus K^{2\eps})$
\end{corollary}




\subsection{Lipschitz estimations for the map $\ell$}

In this paragraph, we bound the Lipschitz constants of the
restriction of the maps $\nabla \d_K$, $\tau$ and (finally)
$\ell$ to the subset $S_\mu^\eps \subseteq \partial K^\eps$.

First, let $\partial K^{\eps,t}$ be the set of points $x$ in $\partial
K^\eps$ where the distance function is differentiable, and such that
$\tau(x)$ is bounded from below by $t$. In particular, notice that
$S_\mu^\eps$ is contaiend in $\partial K^{\eps,\eps}$.  The following
Lemma proves that the functions $\Psi^t$ and $\nabla_x \d_K$ are
Lipschitz on $\partial K^{\eps,t}$:

\begin{lemma} \label{lemma:grad:lipschitz}
\begin{itemize} \item[(i)] The restriction of $\Psi^t$ to $\partial K^{\eps,t}$
is $(1+t/\eps)$-Lipschitz.
\item[(ii)]
The gradient of the distance function, $x \mapsto
\nabla_x \d_K$, is $3/\eps$-Lipschitz on $\partial K^{\eps,\eps}$.
\end{itemize}
\end{lemma}

\begin{proof}
Let $x$ and $x'$ be two points of $\partial K^\eps$ with
$\tau(x),\tau(x') > t$, $p$ and $p'$ their projections on $K$ and $y$
and $y'$ their image by $\Psi^t$. We let $u = 1 + t/\eps$ be the scale
factor between $x-p$ and $y-p$, \ie:
$$ (*)~~~y' - y = u (x' - x) + (1-u) (p' - p)$$
Using the fact that $y$ projects to $p$, and the definition of $u$, we
have:
$$\nr{y - p}^2 \leq \nr{y - p'}^2 = \nr{y-p}^2 + \nr{p - p'}^2 + 2
\sca{y - p}{p - p'}$$
\begin{equation*}
\begin{split}
&\hbox{\ie} 0 \leq \nr{p - p'}^2 + 2 u \sca{x-p}{p-p'} \\
&\hbox{\ie} \sca{p-x}{p  - p'} \leq \frac{1}{2}u^{-1} \nr{p - p'}^2
\end{split}
\end{equation*}
Summing this last inequality, the same inequality with primes and the
equality $\sca{p' -p }{p-p'} = - \nr{p'-p}^2$ gives
$$(**)~~~\sca{x'-x}{p' - p} \leq \left(1 -
  u^{-1}\right)\nr{p'-p}^2$$
Using $(*)$ and $(**)$ we get the  desired Lipschitz inequality
$$
\begin{aligned}
\nr{y - y'}^2 &= u^2 \nr{x - x'}^2 + (1-u)^2 \nr{p'-p}^2 + 2 u(1-u)\sca{x'
  - x}{p' - p} \\
&\leq u^2 \nr{x - x'}^2 - (1-u)^2 \nr{p'-p}^2
\leq \left(1+t/\eps\right)^2 \nr{x - x'}^2
\end{aligned}$$
\end{proof}

The second step is to prove that the restriction of $\tau$ to the set
$S^\eps_\mu$ is also Lipschitz. The technical core of the proof is
contained in the following geometric lemma:

\begin{lemma}
\label{lemma:taulip:computation}
Let $t_0$ denote the intersection time of the ray $x_0 + t v_0$ with
the medial hyperplane $H_{x_0,y_0}$ between $x_0$ and another point
$y_0$, and $t(x,v)$ the intersection time between the ray $x + tv$ and
$H_{x y_0}$. Then, assuming:
\begin{align}
\alpha\nr{x_0-y_0} &\leq \sca{v_0}{x_0-y_0},  \\
\nr{x - y_0} &\leq D, \\
\nr{v - v_0} &\leq \lambda \nr{x - x_0}, \\
\eps &\leq t(x_0,y_0)
\end{align}
one obtains the following bound:
$$
t(x,v) \leq t(x_0, v_0) + \frac{6}{\alpha^2}
(1 + \lambda D) \nr{x - x_0}
$$
as soon as $\nr{x - x_0}$ is small enough (namely, smaller than
$ \eps \alpha^2(1+3 \lambda D)^{-1}$).
\end{lemma}

\begin{proof}
We search the time $t$ such that $\nr{x+tv - x}^2 = \nr{x+tv - y_0}^2$, \ie
$$ t^2\nr{v}^2 = \nr{x-y_0}^2 + 2 t\sca{x-y_0}{v} + t^2\nr{v}^2$$
Hence, the intersection time is $t(x,v) = \nr{x - y_0}^2/2 \sca{y_0 -
  x}{v}$. The lower bound on $t(x_0,y_0)$ translates as
$$ \eps \leq \frac{1}{2} \frac{\nr{x_0 - y_0}^2}{\sca{x_0 - y_0}{v_0}} \leq
\frac{1}{2\alpha} \nr{x_0 - y_0}
$$

 If $\nabla_x t$ and $\nabla_v t$ denote the gradients of
this function in the direction of $v$ and $x$, one has:
\begin{align*}
\nabla_v t(x,v) &= \frac{1}{2} \frac{\nr{x - y_0}^2 (x - y_0)}{\sca{y_0 - x}{v}^2} \\
\nabla_{x} t(x,v) &= \frac{1}{2}\frac{\nr{x - y_0}^2 v + 
2 \sca{y_0 - x}{v} (x - y_0) }{\sca{y_0 - x}{v}^2}
\end{align*}
Now, we bound the denominator of this expression:
\begin{equation*}
\begin{split}
 \sca{x - y_0}{v} &= \sca{x - y_0}{v - v_0} + \sca{x - x_0}{v_0} + \sca{x_0 - y_0}{v_0} \\
&\geq \alpha\nr{x_0 - y_0} - (1 + \lambda  \nr{x - y_0}) \nr{x - x_0}\\
&\geq \alpha\nr{x - y_0} - (2 + \lambda D) \nr{x - x_0}
\end{split}
\end{equation*}
The scalar product $ \sca{x - y_0}{v}$ will be larger than (say)
$\frac{\alpha}{2} \nr{x - y_0}$ provided that
$$ (2 + \lambda D) \nr{x - x_0} \leq \frac{\alpha}{2} \nr{x - y_0} $$
or, bounding from below $\nr{x - y_0}$ by $\nr{x_0 - y_0} - \nr{x_0 - x}
\geq 2\alpha \eps - \nr{x_0 - x}$,
provided that:
$$ (3 + \lambda D) \nr{x - x_0} \leq \alpha^2 \eps
$$
This is the case in particular if $\nr{x - x_0} \leq \alpha^2 \eps 
(3+ \lambda D)^{-1}$. Under that assumption, we have the following bound on the 
norm of the gradient, from which the Lipschitz inequality follows:
$$ \nr{\nabla_{x} t(x,v)} \leq 6/\alpha^2 
\hbox{~~~and~~~}
\nr{\nabla_v t(x,v)} \leq 4D/\alpha^2$$
\end{proof}

Using this Lemma, we are able to show that the function $\ell$ is locally
Lipschitz on the subset $S_\mu^\eps \subseteq \partial K^\eps$:

\begin{proposition}
\label{prop:ell-lipschitz}
The restriction of $\tau$ to $S_\mu^\eps$ is locally $L$-Lipschitz, in
the sense that if $(x,y) \in S_\mu^\eps$ are such that $\nr{x - y} \leq
\delta_0$, then $\nr{\ell(x) - \ell(y)} \leq L \nr{x - y}$ with
$$ L = \BigO\left(\frac{1+\diam(K)/\eps}{(1 - \mu)^{1/2}}\right)\hbox{
  and }\delta_0 = \BigO(\eps/L)$$
\end{proposition}

In order to simplify the proof of this Proposition, we will make use
of the following notation, where $f$ is any function from $X\subseteq
\RR^d$ to $\RR$ or $\RR^d$:
$$\Lip_\delta \restr{f}{X} := \sup \{ \nr{f(x) -
  f(y)}/\nr{x-y}; (x,y) \in X^2 \hbox{ and } \nr{x - y} \leq \delta
\}.$$

\begin{proof}
We start the proof by evaluating the Lipschitz constant of the
restriction of $\tau$ to $S_\mu^\eps$, using Lemma
\ref{lemma:taulip:computation} (Step 1), and then deduce the Lipschitz
estimate for the function $\ell$ (Step 2).

\begin{paragraph}{\textsc{Step 1}}
Thanks to Lemma \ref{lemma:opposite-point}, for any $x$ in
$S_\mu^\eps$, there exists another projection $y$ of $m = \ell(x)$ on
$\partial K^\eps$ such that the cosine of the angle $\theta =
\angle{(x - m, y - m)}$ is at most $\sqrt{(1+\mu^2)/2}$. Let us denote by
$v = \nabla_x \d_K$ the unit vector from $x$ to $m$. The angle between 
$\rvec{yx}$ and $v$ is $\pi/2 - \theta$. Then,
$$\cos(\pi/2 - \theta) = \sin(\theta) = \sqrt{1 - \cos^2(\theta)} \geq
\alpha:=\left(\frac{1-\mu^2}{2}\right)^{1/2}$$

As a consequence, with the $\alpha$ introduced above, one has
$\alpha\nr{x-y} \leq \alpha \abs{\sca{v}{x-y}}$. Moreover, $\nr{x -
  y}$ is smaller than $D = \diam(K^\eps) \leq \diam(K)+\eps$.  For any
other point $x'$ in $S_\mu^\eps$, and $v' = \nabla_{x'} \d_K$, one has
$\nr{v - v'} \leq \lambda \nr{x - x'}$ with $\lambda = 3/\eps$ (thanks
to Lemma \ref{lemma:grad:lipschitz}).

These remarks allow us to apply Lemma
\ref{lemma:taulip:computation}. Using the notations of this lemma, one
sees that $t(x,v)$ is simply $\tau(x)$ while $t(x',v')$ is an upper bound for
$\tau(x')$. This gives us:
\begin{equation*}
\begin{split}
\tau(x') &\leq \tau(x) + 
 \frac{6}{\alpha^2}
(1 + \lambda D) \nr{x - x'} \\
&\leq \tau(x) + M
\nr{x - x'}\\
\hbox{where}~ M &= \BigO\left(\frac{1+\diam(K)/\eps}{\sqrt{1-\mu^2}}\right)
\end{split}
\end{equation*}
as soon as $x'$ is close enough to $x$. From the statement of Lemma
\ref{lemma:taulip:computation}, one sees that $\nr{x - x'} \leq
\delta_0$ with $\delta_0 = \BigO(\eps/M)$ is enough.  Exchanging
the role of $x$ and $x'$, one proves that $\abs{\tau(x) - \tau(x')}
\leq M \nr{x - x'}$, provided that $\nr{x - x'} \leq \delta_0$.
As a conclusion,
\begin{equation}
\Lip_{\delta_0} \left[\restr{\tau}{S_\mu^\eps}\right] =
\BigO\left(\frac{1+\diam(K)/\eps}{\sqrt{1-\mu^2}}\right)
\label{eq:est:1}
\end{equation}
\end{paragraph}

\begin{paragraph}{\textsc{Step 2}}
We can use the following decomposition of the difference
$\ell(x) - \ell(x')$:
\begin{equation}
\ell(x)- \ell(x') = (x' - x) + (\tau(x) - \tau(x')) \nabla_x \d_K +
\tau(x') (\nabla_x \d_K - \nabla_{x'} \d_K)
\label{eq:ell-sum}
\end{equation}
 in order to bound the
(local) Lipschitz constant of the restriction of $\ell$ to
$S_\mu^\eps$ from those computed earlier.
One deduces from this equation that
\begin{equation}
\Lip_{\delta_0} \left[\restr{\ell}{S_\mu^\eps}\right] \leq 1 +
\Lip_{\delta_0} \left[\restr{\tau}{S_\mu^\eps}\right] +
\nr{\tau}_\infty \Lip_{\delta_0} \left[\restr{\nabla \d_K}{S_\mu^\eps}\right]
\label{eq:ell-lip}
\end{equation}
Thanks to Lemma \ref{lemma:dist-K-mumedial}, one has $\abs{\tau(x)} =
\BigO(\diam(K)/(1-\mu)^{1/2})$; combining this with the estimate from Lemma \ref{lemma:grad:lipschitz} that $\Lip \restr{\nabla \d_K}{S_\mu^\eps} \leq 3/\eps$,
this gives 
\begin{equation}
\nr{\tau}_\infty \Lip_{\delta_0} \left[\restr{\nabla \d_K}{S_\mu^\eps}\right] = 
\BigO(\diam(K)/[\eps (1-\mu)^{1/2}])
\label{eq:est:2}
\end{equation}
Putting the estimates \eqref{eq:est:1} and \eqref{eq:est:2} into 
\eqref{eq:ell-lip} concludes the proof.
\end{paragraph}
\end{proof}

In order to be able to deduce Theorem~\ref{th:covering:mumedial} from
Proposition \ref{prop:ell-lipschitz} we need the following bound on
the covering numbers of a levelset $\partial K^r$, where $K$ is any
compact set in $\RR^d$ (see \cite[Proposition~4.2]{ccsm2009boundary}):
\begin{equation}
\LebNum(\partial K^r, \eps) \leq \LebNum(\partial K, r) \LebNum(\Sph^{d-1}, \eps/2r)
\label{eq:cov-offset}
\end{equation}

\begin{proof}[Proof of Theorem \ref{th:covering:mumedial}]
Applying Proposition \ref{prop:ell-lipschitz}, we get the existence
of $$L = \Lip_{\delta_0} \left[\restr{\ell}{S_\mu^\eps}\right] =
\BigO(\diam(K)/(\eps\sqrt{1-\mu})) \hbox{ and } \delta_0 =
\BigO(\eps/L)$$ such that $\ell$ is locally $L$-Lipschitz. In
particular, for any $\eta$ smaller than~$\delta_0$,
\begin{equation}
\begin{split}
\LebNum\left(\Medial_\mu(K) \cap
(\RR^d\setminus K^{\eps}), \eta\right) &= 
\LebNum\left(\ell(S_\mu^\eps), \eta\right) \\
&\leq \LebNum\left(S_\mu^\eps, \eta/L\right)\\
&\leq \LebNum(\partial
K^{\eps/2}, \eta/L).
\end{split}
\label{eq:th:1}
\end{equation}
The bound on the covering number of the boundary of tubular neighborhoods
(equation \eqref{eq:cov-offset}) gives:
\begin{equation}
\label{eq:th:2}
 \LebNum(\partial K^{\eps/2}, \eta/L)  \leq
\LebNum(\partial K,\eps/2) \LebNum\left(\Sph^{d-1}, \frac{\eta}{L \eps}\right).
\end{equation}
Equations \eqref{eq:th:1} and \eqref{eq:th:2}, and the estimation
$\LebNum(\Sph^{d-1}, \rho) \sim \omega_{d-1} \rho^{d-1}$ yield
$$ \LebNum\left(\Medial_\mu(K) \cap
(\RR^d\setminus K^{\eps}), \eta\right) = \LebNum(\partial K,\eps/2)
\BigO\left(\left[\frac{\eta}{L \eps}\right]^{d-1}\right).
$$ Its suffices to replace $L$ by its value from Proposition
\ref{prop:ell-lipschitz} to finish the proof.
\end{proof}

\section{A quantitative stability result for boundary measures}
\label{sec:stability-boundary-attempt}

In this paragraph, we show how to use the bound on the covering
numbers of the $\eps$-away $\mu$-medial axis given in Theorem
\ref{th:covering:mumedial} in order to get a quantitative version of
the $\LL^1$ convergence results for projections. Notice that the
meaning of \emph{locally} in the next statement could also be made
quantitative using the same proof.

\begin{theorem}
The map $K \mapsto \p_K \in \LL^1(E)$ is locally $h$-H\"older for any
exponent $h < \frac{1}{2 (2d-1)}$.
\label{prop:proj-stab}
\end{theorem}

\begin{proof}
Remark first that if a point $x$ is such that $\d_K(x) \leq
\frac{1}{2}L - \dH(K,K')$, then by definition of the Hausdorff
distance, $\d_{K'}(x) \leq \frac{1}{2}L$. In particular, the
orthogonal projections of $x$ on $K$ and $K'$ are at distance at most
$L$. Said otherwise, the set $\Delta_L(K,K')$ is contained in the
complementary of the $\frac{L}{2} - \delta$ tubular neighborhood of
$K$, with $\delta := \dH(K,K')$. As in the previous proof, we will let
$R = \nr{\d_K}_{E,\infty}$, so that $E$ is contained in the tubular
neighborhood $K^R$.

We now choose $L$ to be $\delta^h$, where $h > 0$, and see for which
values of $h$ we are able to get a converging bound. Using Lemma
\ref{lem:delta-included}, we have:
\begin{equation*}
\Delta_L(K,K') \cap K^R \subseteq \left(\Medial_\mu(K) \cap
(\RR^d\setminus K^{\frac{1}{2}(L - \delta) - 2 \sqrt{R\delta}})
\right)^{2\sqrt{R \delta}}
\end{equation*}
For $h < 1/2$, the radius $\frac{1}{2}(L - \delta) - 2 \sqrt{R\delta}$
will be greater than $L/3$ as soon as as soon as $\delta$ is small
enough. Hence,
\begin{equation}
\Delta_L(K,K') \cap K^R\subseteq \left(\Medial_\mu(K) \cap
(\RR^d\setminus K^{L/3}) \right)^{2\sqrt{R \delta}}
\end{equation}
 The $\mu$ above, given by Lemma \ref{lem:delta-included} can then be
 bounded as follows. Note that the constants in the ``big O'' will
 always positive in the remaining of the proof.
\begin{equation*}
\begin{split}
\mu &\leq \left(1 + \left[\frac{L-\delta}{4R}\right]^2\right)^{-1/2} + 4 \sqrt{\delta/L}\\
&= 1 + \BigO(-\delta^{2h} + \delta^{1/2 - h/2})
\end{split}
\end{equation*}
The term will be asymptotically smaller than $1$ provided that $2h <
1/2 - h/2$ \ie $h < 1/5$, in which case $\mu = 1 -
\BigO(\delta^{2h})$. By definition of the covering number, one has:
\begin{equation}
\begin{split}
\Haus^d(\Delta_L(K,K') \cap K^R) &\leq
\Haus^d\left[\left(\Medial_\mu(K) \cap \left(\RR^d \setminus
K^{L/3}\right)\right)^{2 \sqrt{R\delta}}\right]\\
&\leq \LebNum\left(\Medial_\mu(K) \cap \left(\RR^d \setminus
K^{L/3}\right), 2 \sqrt{R\delta}\right) \times \BigO(\delta^{d/2})
\end{split}
\label{eq:quant:I}
\end{equation}
The covering numbers of the intersection $\Medial_\mu(K) \cap
\left(\RR^d \setminus K^{L/3}\right)$ can be bounded using Theorem
\ref{th:covering:mumedial}:
\begin{equation}
\label{eq:quant:II}
\begin{split}
\LebNum&\left(\Medial_\mu(K) \cap \left(\RR^d \setminus K^{L/2}\right), 2
\sqrt{R\delta}\right) \\
&\qquad= \LebNum(\partial K,L/2)\BigO\left(\left[\frac{\diam(K)/\sqrt{R\delta}}{\sqrt{1-\mu^2}}\right]^{d-1}\right) \\
&\qquad= \LebNum(\partial K,L/2) \BigO\left(\delta^{-(h+\frac 1 2)(d-1)}\right)
\end{split}
\end{equation}
Combining equations \eqref{eq:quant:I} and \eqref{eq:quant:II}, and
using the (crude) estimation $\LebNum(\partial K,L/2) = \BigO(1/L^d) =
\BigO(\delta^{-hd})$,
\begin{equation*}
\begin{split}
\Haus^d(\Delta_L(K,K') \cap K^R) &
\leq \LebNum(\partial K,L/2) \BigO(\delta^{-h(d - 1) - \frac{1}{2}(d-1)+\frac{1}{2}d})\\
&\leq  \BigO\left(\delta^{\frac{1}{2} - h(2d - 1)}\right)
\end{split}
\end{equation*}
Hence, following the proof of Proposition \ref{prop:proj-stab-nonquant},
\begin{equation*}
\begin{split}
\nr{\p_{K'} - \p_{K}}_{\LL^1(E)} &\leq L \Haus^d(E) +
\Haus^d(\Delta_{L}(K,K') \cap E) \diam(K \cup K') \\
& = \BigO(\delta^{h} + \delta^{1/2 - h(2d - 1)})
\end{split}
\end{equation*}
The second term converges to zero as $\delta = \d_H(K,K')$ does if $h
< \frac{1}{2(2d - 1)}$. This concludes the proof.
\end{proof}

\bibliographystyle{amsplain}
\bibliography{mumedial}

\end{document}